\documentclass[12pt,twoside]{amsart}
\usepackage{latexsym}
\usepackage{amssymb}
\usepackage{amscd}
\usepackage{amsmath}
\usepackage{amsthm}
\usepackage{enumerate} 
\usepackage{mathrsfs}  
\usepackage{hyperref}
\usepackage[all]{xy}

\title[On log canonical rings]{On log canonical rings} 
\author{Osamu Fujino}
\author{Yoshinori Gongyo}
\date{2013/7/12, version 1.06}
\address{Department of Mathematics, Faculty of Science,
Kyoto University, Kyoto 606-8502, Japan}
\email{fujino@math.kyoto-u.ac.jp}
\address{Graduate School of Mathematical Sciences, 
The University of Tokyo, 3-8-1 Komaba, 
Meguro, Tokyo, 153-8914 Japan.}
\email{gongyo@ms.u-tokyo.ac.jp}
\address{Department of Mathematics, Imperial College London, 
180 Queen's Gate, London SW7 2AZ, UK}
\email{y.gongyo@imperial.ac.uk}

\dedicatory{Dedicated to Professor~Yujiro~Kawamata on the~occasion of his~sixtieth~birthday.}
\newcommand{\Supp}[0]{\operatorname{Supp}}
\newcommand{\Proj}[0]{\operatorname{Proj}}
\newcommand{\Exc}[0]{\operatorname{Exc}}

\newtheorem{thm}{Theorem}[section]
\newtheorem{lem}[thm]{Lemma}
\newtheorem{cor}[thm]{Corollary}

\theoremstyle{definition}

\newtheorem{defn}[thm]{Definition}
\newtheorem{rem}[thm]{Remark}
\newtheorem*{ack}{Acknowledgments}

\newtheorem{say}[thm]{}

\newtheorem{theorema}{Conjecture}

\begin{document}
\bibliographystyle{amsalpha+}

\maketitle

\begin{abstract}
We discuss the  relationship among various conjectures in the minimal 
model theory including the finite generation conjecture of the log canonical rings 
and the abundance conjecture. In particular, we  
show that the finite generation conjecture of the log canonical rings for log 
canonical pairs can be reduced to that of the log canonical rings for purely log terminal 
pairs of log general type.
\end{abstract}

\tableofcontents
\section{Introduction} In this article, 
we discuss the relationship among the following conjectures:

\begin{theorema}\label{conj F} Let $(X,\Delta)$ be a projective log canonical pair 
and $\Delta$ a $\mathbb{Q}$-divisor. Then the log canonical ring 
$$R(X, \Delta):= 
\bigoplus_{m\geq 0}H^0(X, \mathcal{O}_X(\lfloor m(K_X+\Delta ) \rfloor ))$$
 is finitely generated.
\end{theorema}

\begin{theorema}\label{conj F_big}Let $(X,\Delta)$ be a projective purely 
log terminal pair such that $\lfloor \Delta \rfloor$ is irreducible and that 
$\Delta$ is a $\mathbb{Q}$-divisor. 
Suppose that $K_X+\Delta$ is big. Then the log canonical ring 
$$R(X, \Delta)= 
\bigoplus_{m\geq 0}H^0(X, \mathcal{O}_X(\lfloor m(K_X+\Delta ) \rfloor ))$$
 is finitely generated.
\end{theorema}

\begin{theorema}[Good minimal model conjecture]\label{conj M} 
Let $(X,\Delta)$ be a $\mathbb{Q}$-factorial projective 
divisorial log terminal pair and $\Delta$ an $\mathbb{R}$-divisor. 
If $K_X+\Delta$ is pseudo-effective, then $(X, \Delta)$ has a good minimal model. 
\end{theorema}

From now on, Conjecture $\bullet _n$ 
(resp.~Conjecture $\bullet_{\leq n}$) stands for Conjecture $\bullet$  
with $\dim X =n$ (resp.~$\dim X \leq n$). Remark that 
in Conjectures \ref{conj F}, \ref{conj F_big}, and \ref{conj M}
we may assume that $(X,\Delta)$ is log smooth, i.e., 
$X$ is smooth and $\Delta$ has a 
simple normal crossing support by taking suitable resolutions. 

The following result is the main theorem:

\begin{thm}[Main Theorem]\label{main theorem} 
Conjectures \ref{conj F}$_{n}$, \ref{conj F_big}$_{n}$, 
and \ref{conj M}$_{\leq n-1}$ are all equivalent.
\end{thm}

We remark that 
Conjecture \ref{conj F_big}$_n$ implies  Conjecture \ref{conj F}$_{\leq n}$ 
by Theorem \ref{main theorem} because Conjecture \ref{conj F}$_{\leq n-1}$ 
directly follows from Conjecture \ref{conj M}$_{\leq n-1}$. 
We also remark that the equivalence of Conjecture \ref{conj F}$_n$ and 
Conjecture \ref{conj M}$_{\leq n-1}$ seems to be a folklore statement, 
though we have never seen the explicit statement in the literature. 

In \cite{fm}, the first author and Shigefumi Mori 
proved that the finite generation of the log canonical rings for projective klt pairs can be 
reduced to the case when the log canonical divisors are big by using the 
so-called Fujino--Mori canonical bundle formula 
(see \cite[Theorem 5.2]{fm}). 
This reduction seems to be indispensable 
for the finite generation of the log canonical 
rings for klt pairs (see, for example, \cite{bchm}, \cite{lazic}, and so on). 
Unfortunately, the reduction arguments in \cite{fm} can not be 
directly applied to log canonical pairs because 
the usual perturbation techniques do not work 
well for log canonical pairs (cf.~Remark \ref{fm-reduction}). The following statement is 
contained in our main theorem:~Theorem \ref{main theorem}. 

\begin{cor}\label{cor0}Conjecture \ref{conj F_big}$_n$ implies 
Conjecture \ref{conj F}$_n$. 
\end{cor}
 
 Corollary \ref{cor0} is one of the motivations of this paper. 
 The proof of Theorem \ref{main theorem} (and Corollary \ref{cor0}) 
 heavily depends on the recent developments 
 in the minimal model theory after \cite{bchm}, 
 for example, \cite{bir-exiII}, \cite{dhp-ext}, \cite{fg3}, \cite{g4}, \cite{hmx}, 
 and so on. It is completely different from the reduction techniques discussed in \cite{fm}. 
 
 In Conjecture \ref{conj F_big}, we may assume that $X$ is smooth, 
 $\Delta$ has a simple normal crossing support, $\lfloor \Delta \rfloor$ is irreducible, and 
 $K_X+\Delta$ is big. 
 Hence Conjecture \ref{conj F_big} looks more approachable than Conjecture 
 \ref{conj F} from the analytic viewpoint (cf.~\cite{dhp-ext}). 
  
As corollaries of Theorem \ref{main theorem} and its proof,  we can also see the following:

\begin{cor}\label{cor1}Assume that 
Conjecture \ref{conj F_big}$_n$ holds. 
Let $(X,\Delta)$ be an $n$-dimensional 
$\mathbb Q$-factorial projective divisorial log terminal pair 
such that $\Delta$ is a $\mathbb Q$-divisor. 
If  $\kappa(X, K_X+\Delta) \geq 1$, then $(X, \Delta)$ has a good minimal model.
\end{cor}

\begin{cor}\label{cor2}Assume that Conjecture \ref{conj F_big}$_n$ holds.  
Let $(X,\Delta)$ be an $n$-dimensional log canonical pair, $\Delta$ a $\mathbb{Q}$-divisor, 
and $f:X \to S$ a proper morphism 
onto an algebraic variety $S$. Then the relative log canonical ring 
$$R(X/S, \Delta):= 
\bigoplus_{m\geq 0}f_*\mathcal{O}_X(\lfloor m(K_X+\Delta ) \rfloor )$$ is a 
finitely generated $\mathcal{O}_S$-algebra.
\end{cor}

If Conjecture \ref{conj F_big}$_n$ implies Conjecture \ref{conj M}$_n$, 
then Conjectures \ref{conj F}, \ref{conj F_big}, and \ref{conj M} hold 
in any dimension by Theorem \ref{main theorem}. 
Unfortunately, Corollary \ref{cor1} is far from the complete solution of Conjecture 
\ref{conj M}$_n$ under Conjecture \ref{conj F_big}$_n$. 
For the details of Conjecture \ref{conj M}, we recommend the reader to see 
\cite[Section 5]{fg3} (see also Section \ref{sec4}:~Appendix). 

In \cite{fujino-surface}, the first author solved Conjecture \ref{conj F}$_4$. 
Conjecture \ref{conj F}$_n$ with $n\geq 5$ is widely open. 
For surfaces, $R(X, \Delta)$ is known to 
be finitely generated under the assumption that $\Delta$ is a boundary 
$\mathbb Q$-divisor and $X$ is $\mathbb Q$-factorial. 
When $\dim X=2$, 
we do not have to assume that the pair $(X, \Delta)$ is log canonical 
for the minimal model theory. For the details, see \cite{fujino-surface}. 

\begin{ack}
The first author was partially supported 
by the Grant-in-Aid for Young Scientists (A) $\sharp$24684002 from JSPS.  
The second author was partially supported 
by the Grant-in-Aid from JSPS $\sharp$24840009 
and Research expense from the JRF fund. The authors 
would like to thank the referee for useful 
suggestions. In particular, 
he/she kindly informed them of his/her alternative proof of Lemma \ref{lem4}. 
\end{ack} 

We will work over $\mathbb C$, the field of complex numbers, throughout this paper. 
We will make use of the standard notation as in \cite{kamama}, \cite{komo},
\cite{bchm}, \cite{fujino-what} and \cite{fujino-funda}. 

\section{Preliminaries} In this section, we collect together some definitions and notation. 

\begin{say}[Pairs]
A pair $(X, \Delta)$ consists of a normal variety $X$ over 
$\mathbb C$ and an effective $\mathbb R$-divisor 
$\Delta$ on $X$ such that $K_X+\Delta$ is $\mathbb R$-Cartier. 
A pair $(X, \Delta)$ is called {\em{klt}} (resp.~{\em{lc}}) if for any projective birational morphism 
$g:Z\to X$ from a normal variety $Z$, every coefficient of $\Delta_Z$ is $<1$ (resp.~$\leq 1$) where 
$K_Z+\Delta_Z:=g^*(K_X+\Delta)$. Moreover a pair $(X, \Delta)$ is called {\em{canonical}} 
(resp.~{\em{plt}}) 
if for any projective birational morphism 
$g:Z\to X$ from a normal variety $Z$, every coefficient of $g$-exceptional 
components of  $\Delta_Z$ is $\leq 0$ 
(resp.~$<1$). 
Let $(X, \Delta)$ be an lc pair. If there is 
a projective birational morphism $g:Z\to X$ from a smooth projective variety $Z$ such that 
every coefficient of $g$-exceptional 
components of $\Delta_Z$ is $<1$, the exceptional locus $\Exc(g)$ of $g$ is a divisor, and $\Exc (g)\cup 
\Supp \Delta_Z$ is a simple normal crossing divisor on $Z$, then $(X, \Delta)$ is called {\em{dlt}}. 

We note that {\em{klt}}, {\em{plt}}, {\em{dlt}}, and {\em{lc}} 
stand for {\em{kawamata log terminal}}, {\em{purely log terminal}}, {\em{divisorial 
log terminal}}, and {\em{log canonical}}, respectively. 
\end{say}

Let us recall the definition of {\em{log minimal models}}. 
In Definition \ref{real-def-minimalmodelsenseofBS}, all the varieties are assumed to be  
projective. 

\begin{defn}[cf.~{\cite[Definition 2.1]{bir-exiII}}]\label{real-def-minimalmodelsenseofBS}
A pair $(Y, \Delta_Y)$ is a \emph{log birational model} of $(X, \Delta)$ if we are given a birational map
$\phi\colon X\dashrightarrow Y$ and $\Delta_Y=\Delta^\sim+E$ where $\Delta^\sim$ is 
the birational transform of $\Delta$ and 
$E$ is the reduced exceptional divisor of $\phi^{-1}$, that is, $E=\sum E_j$ where $E_j$ is 
a prime divisor on $Y$ which is exceptional over $X$ for every $j$. 
A log birational model $(Y, \Delta_Y)$ is a \emph{nef model} 
of $(X, \Delta)$ if in addition
\begin{itemize}
\item[(1)] $(Y, \Delta_Y)$ is $\mathbb{Q}$-factorial dlt, and 
\item[(2)] $K_Y+\Delta_Y$ is nef. 
\end{itemize}
And  we call a nef model $(Y, \Delta_Y)$ 
a \emph{log minimal model of $(X, \Delta)$ {\em{(}}in the sense of Birkar--Shokurov{\em{)}}} if in addition
\begin{itemize}
\item[(3)] for any prime divisor $D$ on $X$ which is exceptional over $Y$, we have
$$
a(D,X,\Delta)<a(D,Y,\Delta_Y).
$$
\end{itemize}

Let $(Y, \Delta_Y)$ be a log minimal model of $(X, \Delta)$. 
If $K_Y+\Delta_Y$ is semi-ample, then $(Y, \Delta_Y)$ is called a {\em{good 
minimal model}} of $(X, \Delta)$. 
\end{defn}
When $(X,\Delta)$ is plt, 
a log minimal model of $(X, \Delta)$ in the sense of Birkar--Shokurov is a log minimal model 
in the traditional sense (see \cite{komo} and \cite{bchm}), that is, $\phi\colon X\dashrightarrow Y$ 
extracts no divisors. 
For the details, see \cite[Remark 2.6]{bir-I}. 

\begin{rem}\label{2323}
Assume that Conjecture \ref{conj M}$_{\leq n}$ holds. 
Let $(X, \Delta)$ be a projective $\mathbb Q$-factorial 
dlt pair with $\dim X=n$ such that $K_X+\Delta$ is pseudo-effective. 
Then, by \cite[Corollary 1.6]{bir-exiII}, there is a sequence of divisorial 
contractions and flips starting with $(X, \Delta)$ 
and ending up with a good minimal model $(Y, \Delta_Y)$. 
In particular, $X\dashrightarrow Y$ extracts no divisors. 
Therefore, $(Y, \Delta_Y)$ is a log minimal model of $(X, \Delta)$ in the traditional sense. 
\end{rem}

\section{Proof of Main Theorem}For the proof of 
the main theorem:~Theorem \ref{main theorem}, 
we discuss the relationship among the following conjectures:

\begin{theorema}[Abundance conjecture]\label{conj A} 
Let $(X, \Delta)$ be a projective log canonical pair. 
If $K_X+\Delta$ is nef, then $K_X+\Delta$ is semi-ample.
\end{theorema}

\begin{theorema}[Non-vanishing conjecture]\label{conj N} 
Let $(X,\Delta)$ be a 
projective log canonical pair.  If $K_X+\Delta$ is pseudo-effective, 
then there exists 
some effective $\mathbb{R}$-divisor $D$ such that 
$D \sim_{\mathbb{R}} K_X+\Delta$.
\end{theorema}

\begin{theorema}[Non-vanishing conjecture for smooth 
varieties]\label{conj N_sm} Let $X$ be a 
smooth projective variety.  If $K_X$ is pseudo-effective, then there exists 
some effective $\mathbb{Q}$-divisor $D$ such that $D \sim_{\mathbb{Q}} K_X$.
\end{theorema}

For the above conjectures, we show the following lemmas:

\begin{lem}\label{lem3} 
Conjecture \ref{conj F_big}$_n$ and 
Conjecture \ref{conj N}$_{\leq {n-1}}$ imply Conjecture \ref{conj A}$_{\leq n-1}$.
\end{lem}

\begin{lem}\label{lem4} 
Conjecture \ref{conj F_big}$_n$ 
implies Conjecture \ref{conj N_sm}$_{\leq n-1}$.
\end{lem}

\begin{lem}\label{lem5} 
Conjecture \ref{conj N_sm}$_{\leq n}$ and 
Conjecture \ref{conj A}$_{\leq n-1}$ imply Conjecture \ref{conj N}$_{\leq {n}}$.
\end{lem}

\begin{lem}\label{lem6} 
Conjecture \ref{conj F_big}$_n$ implies Conjecture \ref{conj M}$_{\leq n-1}$.
\end{lem}

\begin{lem}\label{lem7} 
Assume that Conjecture \ref{conj M}$_{\leq n-1}$ holds. 
Let $(X,\Delta)$ be an $n$-dimensional $\mathbb Q$-factorial 
projective divisorial log terminal pair such that $\Delta$ is a $\mathbb{Q}$-divisor and  
$\kappa(X, K_X+\Delta) \geq 1$. Then $(X, \Delta)$ has a good minimal model. 
In particular, 
Conjecture \ref{conj M}$_{\leq n-1}$ implies Conjecture \ref{conj F}$_{\leq n}$.
\end{lem}

Let us start the proof of the lemmas. 

\begin{proof}[Proof of Lemma \ref{lem3}] 
By taking a dlt blow-up and using 
Shokurov polytope (cf.~\cite[Proposition 3.2.~(3)]{bir-exiII} 
and \cite[Theorem 18.2]{fujino-funda}), we may assume that 
$(X,\Delta)$ is a $\mathbb{Q}$-factorial dlt pair and that 
$\Delta$ is a $\mathbb{Q}$-divisor. Moreover by taking a product with an Abelian variety 
we may further 
assume $\dim X=n-1$.  The abundance conjecture 
follows from \cite[Theorem A.6]{lazic} and 
Conjecture \ref{conj F_big}$_n$ 
when $(X,\Delta)$ is klt. 
For a log canonical pair $(X,\Delta)$ with nef $K_X+\Delta$, its semi-ampleness 
follows from \cite[Theorem 5.5]{fg3} by Conjecture \ref{conj N}$_{\leq n-1}$ and 
the abundance theorem for klt pairs established above.
\end{proof}

\begin{proof}[Proof of Lemma \ref{lem4}] We may assume 
$\dim X=n-1$ by taking a product with an Abelian variety.
Let $X \subset \mathbb{P}^N $ be a projectively 
normal embedding. We consider the $\mathbb P^1$-bundle 
$$p: Y:=\mathbb{P}_X(\mathcal{O}_X\oplus \mathcal{O}_X(-1)) \to X.$$
Let $f:Y \to Z$ be the birational contraction of the negative section $E$ on $Y$ and $H$ a 
general sufficiently ample $\mathbb{Q}$-divisor on $Z$ such that $\lfloor H \rfloor=0$ and 
$K_Y+E+f^*H$ is big. Set $\Delta_Y= E+f^*H$. 
Without loss of generality, we may assume 
that $(Y, \Delta_Y)$ is a canonical pair 
with 
$\lfloor  E+f^*H \rfloor=E$. By the assumption 
(Conjecture \ref{conj F_big}$_n$), $R(Y, \Delta_Y)$ is finitely generated. 
Then $(Y^\dag, \Delta_{Y^\dag})$, where $Y^\dag=\Proj R(Y, \Delta_Y)$ and 
$\Delta_{Y^\dag}$ is the pushforward of $\Delta_Y$ on $Y^\dag$, is the log canonical 
model of $(Y, \Delta_Y)$ (see, for example, \cite[Theorem 0-3-12]{kamama}). 
By taking a suitable dlt blow-up of $(Y^\dag, \Delta_{Y^\dag})$, we obtain 
a good minimal model $(Y', \Delta_{Y'})$ of $(Y, \Delta_Y)$ (cf.~\cite{bir-I}). 
See also \cite[Theorem 3.7]{bir-exilcflip}. 
Note that $\varphi: Y\dashrightarrow Y'$ extracts no divisors since $(Y, \Delta_Y)$ is plt. 
Moreover, we may assume that this birational map
$$\varphi:Y \dashrightarrow Y'$$
is a composition of 
$(K_Y+\Delta_Y)$-flips and $(K_Y+\Delta_Y)$-divisorial 
contractions by \cite[Corollary 2.9]{hacon-xu-lc-closure}. 
We note that 
$E$ is not contracted by $\varphi$. Indeed,  if $E$ is contracted, then $E$ is uniruled. 
However,  by \cite[0.3 Corollary]{bdpp}, 
$E$ is not uniruled since $K_E$ 
is pseudo-effective. Note that $E\simeq X$. 
Now we see that $K_{Y'}+\Delta_{Y'}$ is semi-ample 
by the finite generation of 
$R(Y', \Delta_{Y'})$, where $\Delta_{Y'}= \varphi_*\Delta_Y$. 
Take a general member $D' \in |m(K_{Y'}+\Delta_{Y'})|$ 
such that $\varphi_* E\not \subset \Supp D'$ for some sufficiently 
divisible positive integer $m$. 
Then $D'$ induces some effective 
$\mathbb{Q}$-divisor $D$ such that 
$D \sim_{\mathbb{Q}}K_Y+\Delta_Y$ and $E\not \subset \Supp D$. 
Thus we can see $\kappa(X, K_X)=\kappa(E, K_E) \geq 0$ since
 $$K_E = (K_Y+\Delta_Y)|_{E} \sim_{\mathbb{Q}} D|_{E} \geq 0.$$
 Therefore, we obtain Conjecture \ref{conj N_sm}$_{\leq n-1}$. 
\end{proof}

The following proof is pointed out by the referee:

\begin{proof}[Alternative proof of Lemma \ref{lem4}] 
Let $Y$, $Z$ and $E$ be as in the above proof of Lemma \ref{lem4} 
and let $A$ be an ample Cartier divisor such that $\mathcal{O}_X(1)\simeq 
\mathcal{O}_X(A)$. 
Since $K_X$ is pseudo-effective, $K_X+A$ is big. 
Let $H'$  be a hyperplane on $Z\subset \mathbb P^{N+1}$. 
Then we can easily check that 
$$
K_Y+E+2f^*H' \sim E +p^*(K_X+A). 
$$
Let $\widetilde H$ be a $\mathbb Q$-Cartier $\mathbb Q$-divisor 
on $Z$ such that $2\widetilde H$ is a general 
member of $|4H'|$. 
Then $(Y, E+f^*\widetilde H)$ is canonical, $\lfloor E+f^*\widetilde H\rfloor=E$,  
and 
$$
K_Y+E+f^*\widetilde H \sim_{\mathbb Q} E +p^*(K_X+A). 
$$
It is easy to see that $K_Y+E+f^*\widetilde H$ is big. 
Therefore, 
$R(Y, E+f^*\widetilde H)$ is finitely generated 
by the assumption (Conjecture \ref{conj F_big}$_n$). 
Since  $\mathcal{O}_Y(E +p^*(K_X+A))$ is the tautological 
line bundle associated to the $\mathbb P^1$-bundle 
$\mathbb P_X(\mathcal{O}_X(K_X)\oplus 
\mathcal{O}_X(K_X+A))\to X$, the finite generation of $R(Y, E+f^*\widetilde H)$ 
is equivalent to that of  
$$R(X;K_X, K_X+A):=
\bigoplus _{m_1, m_2\geq 0}H^0(X, \mathcal O_X(m_1K_X+
m_2(K_X+A))). 
$$
By \cite[Theorem 3]{cl}, this implies that $\kappa (X, K_X)\geq 0$ since 
$K_X$ is pseudo-effective and $K_X+A$ is big. 
\end{proof} 

\begin{proof}[Proof of Lemma \ref{lem5}] This follows from 
\cite[Theorem 8.8]{dhp-ext} and \cite[Theorem 1.5]{g4}. Note 
that we can use the ACC theorems in \cite{hmx}.
\end{proof}

\begin{proof}[Proof of Lemma \ref{lem6}] 
By \cite{bir-exiII}, it is enough to show 
Conjecture \ref{conj A}$_{\leq {n-1}}$ and 
Conjecture \ref{conj N}$_{\leq {n-1}}$. We show these 
conjectures by induction on the dimension. 
Now we assume that 
Conjecture \ref{conj A}$_{\leq {d-1}}$ and 
Conjecture \ref{conj N}$_{\leq {d-1}}$ hold for $d<n$. 
Note that Conjecture \ref{conj N_sm}$_{\leq n-1}$ holds 
by Lemma \ref{lem4}. By Lemma \ref{lem5}, 
Conjecture \ref{conj N}$_{\leq d}$ holds. On the other hand, 
by Lemma \ref{lem3} and its proof, Conjecture \ref{conj A}$_{\leq {d}}$ holds. 
Thus we see that Conjecture \ref{conj A}$_{\leq {n-1}}$ and 
Conjecture \ref{conj N}$_{\leq {n-1}}$ hold.  
\end{proof}

\begin{rem}\label{rem36} 
By \cite{bir-exiII}, Conjecture \ref{conj A}$_{\leq n}$ and Conjecture 
\ref{conj N}$_{\leq n}$ imply Conjecture \ref{conj M}$_{\leq n}$. 
This fact was used in the proof of Lemma \ref{lem6}. On the other hand, it is easy 
to see that Conjecture \ref{conj M}$_{\leq n}$ implies Conjecture  
\ref{conj A}$_{\leq n}$ and Conjecture \ref{conj N}$_{\leq n}$ by using dlt blow-ups. 
See also  Remark \ref{2323}. 
\end{rem}

\begin{proof}[Proof of Lemma \ref{lem7}] By 
\cite{bir-exiII}, we may assume that $K_X+\Delta$ is nef. 
By 
\cite[Proposition 3.1]{fukuda-num-eff} (cf.~\cite[Theorem 7.3]{kawamata_pluri}), 
we obtain that $K_X+\Delta$ is 
abundant, i.e.~$\kappa(X, K_X+\Delta)=\nu(X, K_X+\Delta)$, 
where $\nu(\bullet)$ is the numerical dimension (see, for example, 
\cite[Lemma 6-1-1]{kamama}). Thus  
we see that $K_X+\Delta$ is semi-ample by \cite[Theorem 4.6]{fg3}. 
\end{proof}

Now we give the proof of Theorem \ref{main theorem}. 

\begin{proof}[Proof of Theorem \ref{main theorem}] It is obvious that 
Conjecture \ref{conj F}$_{n}$ implies Conjecture \ref{conj F_big}$_{n}$. 
By Lemma \ref{lem6}, Conjecture \ref{conj F_big}$_{n}$ 
implies Conjecture \ref{conj M}$_{\leq n-1}$. 
By Lemma \ref{lem7}, Conjecture \ref{conj M}$_{\leq n-1}$ implies 
Conjecture \ref{conj F}$_{\leq n}$. Thus we finish the proof 
of Theorem \ref{main theorem}.
\end{proof}

Finally, we discuss the corollaries. 
Corollary \ref{cor0} is contained in Theorem \ref{main theorem}.  
Corollary \ref{cor1} is a direct consequence of Lemma \ref{lem6} and Lemma \ref{lem7}.
The proof of \cite[Theorem 1.1]{fujino-f.g.} works for Corollary \ref{cor2}. 
Note that \cite{fujino-f.g.} depends on \cite{bir-I} and \cite{fujino-ab}. 
Now we can use more powerful results in \cite{bir-exiII} and \cite{fg3}. 

We close this section with a remark on \cite[Theorem 5.2]{fm}. 

\begin{rem}\label{fm-reduction}
Let $(X, \Delta)$ be a projective log canonical pair such that 
$\Delta$ is a $\mathbb Q$-divisor. 
Let $\Phi:X\dashrightarrow Z$ be the Iitaka fibration with 
respect to $K_X+\Delta$. 
By taking a suitable resolution, 
we assume that $\Phi$ is a morphism, $X$ is smooth, and $\Supp\Delta$ is a simple 
normal crossing divisor on $X$. 
Suppose that every log canonical center of $(X, \Delta)$ is dominant 
onto $Z$. 
Then $R(X, \Delta)$ is finitely generated. 

By using a generalization of the semi-positivity theorem 
(see \cite[Theorem 3.9]{fujino-high} and \cite[Theorem 3.6]{fg-moduli}), 
we can formulate a canonical bundle formula 
for log canonical pairs as in \cite[Section 4]{fm}. 
By using the canonical bundle formula for log canonical 
pairs, the proof of \cite[Theorem 5.2]{fm} 
works for the above setting. 
We leave the details as exercises for the reader. 
Note that the finite generation of the log canonical rings for projective klt pairs holds 
by \cite{bchm}. 
\end{rem}

\section{Appendix}\label{sec4}

In this appendix, we discuss Conjecture \ref{conj M}. 
The results in this appendix are essentially contained in \cite[Section 5]{fg3}. 

Let us recall the following conjecture (see \cite[Conjecture 1.3]{dhp-ext} 
and \cite[Conjecture 1.10]{fg3}). 

\begin{theorema}[DLT extension conjecture] \label{conjG}
Let $(X, \Delta)$ be a projective divisorial log terminal pair such that 
$\Delta$ is a $\mathbb Q$-divisor, 
$\lfloor \Delta\rfloor=S$, $K_X+\Delta$ is nef, and $K_X+\Delta\sim _{\mathbb Q}D\geq 0$ where 
$S\subset \Supp D$. 
Then the restriction map 
$$
H^0(X, \mathcal O_X(m(K_X+\Delta)))\to H^0(S, \mathcal O_S(m(K_X+\Delta)))
$$ 
is surjective for all sufficiently divisible integers $m\geq 2$. 
\end{theorema} 

\begin{thm}\label{thma1} 
Conjecture \ref{conj N_sm}$_{\leq n}$ and Conjecture \ref{conjG}$_{\leq n}$ imply Conjecture \ref{conj M}$_{\leq n}$.  
\end{thm} 

\begin{proof}
By induction on the dimension, we may assume that 
Conjecture \ref{conj M}$_{\leq n-1}$ holds true. Therefore, we obtain Conjecture \ref{conj A}$_{\leq n-1}$ 
(cf.~Remark \ref{rem36}). 
Lemma \ref{lem5}, Conjecture \ref{conj N_sm}$_{\leq n}$, and Conjecture \ref{conj A}$_{\leq n-1}$ imply Conjecture \ref{conj N}$_{\leq n}$. 
Finally, by \cite[Theorem 5.9 and Corollary 5.10]{fg3}, Conjecture \ref{conj N}$_{\leq n}$ 
and Conjecture \ref{conjG}$_{\leq n}$ imply Conjecture \ref{conj M}$_{\leq n}$. 
\end{proof}

We note that, for Theorem \ref{thma1}, it is sufficient to prove Conjecture \ref{conjG} under 
the extra assumptions:~$X$ is $\mathbb Q$-factorial, 
$\kappa (X, K_X+\Delta)=0$, 
and $\Supp D\subset \Supp \Delta$. For the details, see the 
proof of \cite[Theorem 5.9]{fg3}. 

\begin{rem}
Conjecture \ref{conjG} holds true if $K_X+\Delta$ is semi-ample 
(see \cite[Proposition 5.12]{fg3}). 
Therefore, Conjecture \ref{conjG} follows from Conjecture \ref{conj A}. 
\end{rem}

\begin{rem}
In \cite[Conjecture 1.3]{dhp-ext}, it is assumed that $$S\subset \Supp D\subset \Supp \Delta$$ 
in Conjecture \ref{conjG}. 
\end{rem}

\begin{theorema}[Abundance conjecture for klt pairs with $\kappa =0$]\label{conjH} 
Let $(X, \Delta)$ be a projective kawamata log terminal pair such that 
$\Delta$ is a $\mathbb Q$-divisor with $\kappa (X, K_X+\Delta)=0$. 
Then $\kappa _{\sigma}(X, K_X+\Delta)=0$, where 
$\kappa _{\sigma}$ denotes Nakayama's numerical dimension.  
\end{theorema}

\begin{rem}
It is known that the condition $\kappa _{\sigma}(X, K_X+\Delta)=0$ 
is equivalent to the existence of good minimal models of $(X, \Delta)$ 
(see, for example, \cite{druel} and \cite{gongyo-minimal}). 
\end{rem} 

We can easily check the following statement (cf.~the proof of \cite[Theorem 5.9]{fg3}). 

\begin{thm}\label{thma2}
Conjecture \ref{conj N_sm}$_{\leq n}$ and Conjecture \ref{conjH}$_{\leq n}$ imply Conjecture \ref{conj M}$_{\leq n}$.  
\end{thm} 

We leave the details as exercises for the reader. 
The proof of Theorem \ref{thma2} is almost the same as that of Theorem \ref{thma1}. 


\end{document}